\documentclass{article}
\usepackage{graphicx} 
\usepackage{amsfonts}
\usepackage{amsmath}
\usepackage{amssymb}
\usepackage{amsthm}
\usepackage{todonotes}
\usepackage{url}
\newtheorem{lemma}{Lemma}
\newtheorem{defini}{Definition}
\newtheorem{corollary}{Corollary}
\newtheorem{thm}{Theorem}
\newtheorem{remark}{Remark}
\newtheorem{prop}{Proposition}

\usepackage{url}

\title{Recognizable Realizability}
\author{Merlin Carl}

\begin{document}

\maketitle

\begin{abstract}
We introduce a notion of realizability with ordinal Turing machines based on recognizability rather than computability, i.e., the ability to uniquely identify an object. We show that the arising concept of $r$-realizabilty has the property that all axioms of Kripke-Platek set theory are  $r$-realizable and that the set of $r$-realizable statements is closed under intuitionistic provability.
\end{abstract}

\section{Introduction}

Kleene realizability yields a natural interpretation to the intuitive conception that a (number-theoretical) statement is \textit{effectively true}. An extension of Kleene realizability to transfinite computablity was explored in \cite{CGP}, where an application to intuitionistic set theory was provided, another one appearing in \cite{P}. The guiding idea of this this notion of realizability is that, in order for a statement to be effectively true, one needs to be able to \textit{construct} evidence for it. Thus, for example, to realize an statement of the form $\forall{x}\exists{y}\psi(x,y)$, one needs to be able to compute a suitable $y$ for any given $x$. 

A natural epistemic alternative to the requirement of construction would be the ability to uniquely \textit{identify} a witness. This variant is closely related to the so-called \textit{existence property} of intutionistic logic (see, e.g., \cite{Rathjen}), which states that, if there is some $x$ such that $\phi(x)$, then there is a formula $\psi$ such that there is a unique $y$ with $\psi(y)$ and this $y$ also satisfies $\phi(y)$. 

In ordinal computability (\cite{CarlBuch}), the obvious way to model the ability to uniquely identify a witness is \textit{recognizability}: A set $x$ of ordinals is called \textit{recognizable} if and only if there is a program $P$ such that $P^{y}$ halts with output $1$ if and only if $y=x$, while otherwise, it halts with output $0$. This motivates the consideration of realizability notions based on recognizability rather than computability. (This approach was proposed in \cite{Ca2018}, section $7$). 

In this note, we will define such a notion, called \textit{recognizable realizability} or $r$-realizability, and prove that, like the OTM-realizability defined in \cite{CGP}, it has several desirable properties for a notion of constructive set-theoretical truth:

\begin{enumerate}
\item All axioms of KP are $r$-realizable.
\item If a formula $\psi$ follows intuitionistically from a set of formulas $\Phi$, and all elements of $\Phi$ are $r$-realizable, then so is $\psi$. 
\item There are instances of $\forall{x}(\phi(x)\vee\neg\phi(x))$ that are not $r$-realizable.
\item Impredicative instances of the separation axiom are in general not $r$-realizable.
\item The power set axiom is not $r$-realizable.
\end{enumerate}

$r$-realizability thus shares the crucial properties of OTM-realizability \cite{CGP}, \cite{CarlNote} and randomized OTM-realizability \cite{Ca2024}, contributing to the picture that these are stable features of realizability approaches to set theory based on transfinite machine models.

\section{Definition and Basic Techniques}

We will encode arbitrary sets as sets or ordinals in the usual way: For a set $x$, if $f:\alpha\rightarrow\text{tc}(\{x\})$ is a bijection between an ordinal $\alpha$ and the transitive closure of $\{x\}$ satisfying $f(0)=x$, the set $c_{f}:=\{p(\iota,\xi):f(\iota)\in f(\xi)\}$ is the code of $x$ corresponding to $f$, where $p$ denotes Cantor's ordinal pairing function. Following \cite{CGP}, if $c\subseteq\text{On}$ encodes a set, we write $\text{decode}(c)$ for this set. If $x\in\text{tc}(y)$ and $c_{y}\subseteq\text{On}$ is a code for $y$, then there is a natural code for $c_{x}$ that can be derived from $c_{y}$, by letting $\alpha^{\prime}$ be the order type of $f^{-1}[\text{tc}(x)]$, $o:\alpha^{\prime}\rightarrow f^{-1}[\text{tc}(x)]$ the order isomorphism, $f^{\prime}(\xi)=f(o(\xi))$ for $\xi\in f^{-1}[\text{tc}(x)]$ and $c_{x}$ the code for $x$ obtained from $f^{\prime}$. We will refer to this code as the code for $x$ \text{derived} from $c_{y}$ and denote it by $c_{x}^{c_{y}}$. It is easy to see that $c^{c_y}_{x}$ can be OTM-computed uniformly in $c_{y}$ and any code for $x$.

The definition of ordinal Turing machines (OTMs) can be found in Koepke \cite{Koepke}. It is easy to see that OTMs can decide membership and equality between encoded sets, and in fact decide the truth of arbitrary $\in$-sentences when all quantifiers are restricted to a set $x$ relative to any code for $x$ (cf. \cite{Koepke}). We use the term OTMs when ordinal parameters are not allowed (i.e., if computations need to start with nothing besides the input on the tape); otherwise, we speak of parameter-OTMs (pOTMs) (for which one may also initially mark some ordinal as a part of the program). 


When $a$ and $b$ are sets of ordinals, we write $a\oplus b$ for the set $\{2\iota:\iota\in a\}\cup\{2\iota+1:\iota\in b\}$. Conversely, if $a$ is a set of ordinals, we write $(a)_{0}$ and $(a)_{1}$ for the unique sets of ordinals satisfying $a=(a)_{0}\oplus(a)_{1}$. In oracles and computation inputs, we will occasionally write $x,y$ where we really mean $x\oplus y$ to simplify the notation. We also use $\delta_{x,y}$ for the (class) function that is $1$ if and only if $x=y$ and otherwise $0$.



We recall the definitions of (relative) (p)OTM-recognizability from \cite{CSW}, adding a notion of recognizable class functions. We give the definitions for pOTMs, the parameter-free version being obtained from the one given below by setting the parameter to $0$.

\begin{defini}{\label{def recog}}
\begin{enumerate}
\item When $a,b\subseteq\text{On}$, we say that $a$ is pOTM-recognizable relative to $b$, written $a\leq_{r}b$, if and only if there are a pOTM-program $P$ and an ordinal $\rho$ such that, for all $x\subseteq\text{On}$, we have $P^{b\oplus x}(\alpha)\downarrow=\delta_{a,x}$. In this case, we also say that $P^{b}(\alpha)$ recognizes $a$. If we can take $b=0$, we say that $a$ is pOTM-recognizable.
\item A class function $f:\text{On}\rightarrow V$ is pOTM-recognizable if and only if there are an OTM-program $P$ and an ordinal $\rho$ such that, for all $\alpha$, $P(\rho,\alpha)$ recognizes a set $c$ of ordinals such that $\text{decode}(c)=f(\alpha)$. 
\item If we allow arbitrary sets of ordinals, rather than single ordinals, as parameters in the last definition, we obtain strong OTM-recognizability (sOTM-recognizability); that is, $f:\text{On}\rightarrow V$ is sOTM-recognizable if and only if there are an OTM-program $P$, and a set $p\subseteq\text{On}$ such that, for all $\alpha\in\text{On}$, $P^{p}(\alpha)$ recognizes a set $c$ of ordinals with $\text{decode}(c)=f(\alpha)$.
\end{enumerate}
\end{defini}

It is not hard to see that relative recognizability is not transitive (see \cite{CSW}, Lemma 3.16). This is at the same time a technical inconvenience for and a conceptual objection against simply replacing computability with recognizability in the definition of realizability: On the technical side, transitivity is heavily used, e.g. in our proof that the $\in$-induction is $r$-realizable. 
On the conceptual side, recall that the idea was to take objects as epistemically accessible when there is a method to uniquely identify them. But then, there are good reasons to ask for a transitive notion: To repeat an example used in the introduction of \cite{CSW}, I should regard a radioactive stone as recognizable because I can recognize a Geiger counter and then recognize the stone using the Geiger counter, even though I have no direct way of recognizing the stone. We thus follow the main idea of \cite{CSW} and close under relative recognizability:

\begin{defini}
\item Let $x,y\subseteq\text{On}$. We say that $x$ is chain-pOTM-recognizable relative to $y$ if and only if there is a finite sequence $y=x_{0},x_{1},...,x_{k}=x$ of sets of ordinals such that, for each $i\in\{1,...,k\}$, we have that $x_{i-1}$ is recognizable relative to $x_{i}$. 

\end{defini}

We recall the following observation from \cite{CSW} (see \cite{CSW}, Lemma 2.7):

\begin{prop}
Let $x,y\subseteq\text{On}$. Then $x$ is chain-pOTM-recognizable relative to $y$ if and only if there is $z\subseteq\text{On}$ such that $z$ is recognizable relative to $y$ and $x=(z)_{0}$.
\end{prop}
\begin{proof}
    See \cite{CSW}, Lemma 2.7: If $x_{0}x_{1}...x_{k}$ is the chain witnessing the $c$-pOTM-recognizability of $x$ from $y$, let $z:=x_{k}\oplus\bigoplus_{i=0}^{k-1}x_{i}$.
\end{proof}

This motivates the following definition:

\begin{defini}

Let $x,y,\rho\subseteq\text{On}$, and let $P$ be an OTM-program. We say that $P$ (\textit{$c$-recognizes}) $x$ relative to $y$ in the parameter $\rho$ if and only if there exists $z\subseteq\text{On}$ such that $P$ pOTM-recognizes $z$ relative to $y$ in the parameter $\rho$ and $x=(z)_{0}$. In this case, we also say that $x$ is $c$-recognizable from $y$ and write $x\leq_{r}^{c}y$.
\end{defini}

\begin{remark}
\begin{itemize}
\item It is easy to see that recognizability implies $c$-recognizability. In many of the proofs of realizability statements below, we will 
prove recognizability instead of $c$-recognizability and only use $c$-recognizability when it is actually necessary.
\item It is also easy to see that $c$-recognizability is transitive (cf. \cite{CSW}): If $x\leq_{r}^{\text{c}}y\leq_{r}^{\text{c}}z$, then there are $a$, $b$ such that $a\leq_{r}y$, $x=(a){0}$, $b\leq_{r}z$ and $y=(b)_{0}$. But then, it is not hard to see that $x\leq_{r}x\oplus (a\oplus b)\leq_{r}z$ and $x=(x\oplus(a\oplus b))_{0}$, so $x\leq_{r}^{c}z$.
\end{itemize}
\end{remark}

\begin{defini}{\label{rho def}}
For an OTM-program $P$ and a parameter $q$, we denote by $\rho(P,q)$ the set of ordinals recognized by $P$ in the parameter $q$, if it exists; otherwise, $\rho(P,q)$ is undefined. 
Moreover, if $\rho(P,q)$ exists, let us write $\rho_{0}(P,q):=(\rho(P,q))_{0}$ and $\rho_{1}(P,q):=(\rho(P,q))_{1})$ for the first and second component, respectively. Thus, if defined, then $\rho_{0}(P,q)$ is the set of ordinals that is $c$-recognized by 
\end{defini}

We now define $r$-realizability. The definition follows the example of Kleene realizability \cite{Kleene}, which was adapted to a concept of OTM-realizability using computability in \cite{CGP}. We give the definitions for the parameter-case, the parameter-free case being obtained by simply setting the parameter equal to $0$. 

\begin{defini}{\label{def r-realizability}}
Let $\phi$ be an $\in$-formula (possibly using set parameters, which will only be made explicit when necessary), $r$ a set. We say that $r$ $r$-realizes $\phi$, written $r\Vdash_{r}^{p}\phi$, if and only if the following holds:

\begin{enumerate}
\item If $\phi$ is atomic (i.e., of the form $a\in b$ or $a=b$), then $r\Vdash_{r}^{p}\phi$ if and only if $\phi$ is true. 
\item If $\phi$ is of the form $\psi_{0}\wedge\psi_{1}$, then $r$ is an ordered pair $(r_{0},r_{1})$ and we have $r_{0}\Vdash_{r}^{p}\psi_{0}$ and $r_{1}\Vdash_{r}^{p}\psi_{1}$.
\item If $\phi$ is of the form $\psi_{0}\vee\psi_{1}$, then $r$ is an ordered pair $(i,r^{\prime})$, where $i\in\{0,1\}$ and $r^{\prime}\Vdash_{r}^{p}\psi_{i}$. 
\item If $\phi$ is of the form $\neg\psi$, then $r\Vdash_{r}^{p}\psi\rightarrow 0=1$. 
\item If $\phi$ is of the form $\psi_{0}\leftrightarrow\psi_{1}$, then $r\Vdash_{r}^{p}((\psi_{0}\rightarrow\psi_{1})\wedge(\psi_{1}\rightarrow\psi_{0}))$. 
\item If $\phi$ is of the form $\psi_{0}(\vec{p})\rightarrow\psi_{1}(\vec{q})$, then $r$ is an ordered pair $(P,\alpha)$ consisting of a pOTM-program $P$ and a parameter $\alpha$ such that, for every all codes $c_{\vec{p}}$, $c_{\vec{q}}$ for $\vec{p}$ and $\vec{q}$, respectively, and every 
$r^{\prime}$ with $r^{\prime}(\vec{p})\Vdash_{r}^{p}\psi_{0}$, $P^{r^{\prime},\vec{p},\vec{q}}(\alpha)$ $c$-recognizes some $r^{\prime\prime}$ such that $r^{\prime\prime}\Vdash_{r}^{p}\psi_{1}(\vec{q})$ in the parameter $c_{\vec{p}}$ and $c_{\vec{q}}$. 
\item If $\phi$ is of the form $\exists{x}\psi(x,\vec{p})$, then $r$ is an ordered pair $(P,\alpha)$ consisting of an OTM-program $P$ and an ordinal $\alpha$ such that, $P^{\vec{p}}(\alpha)$ c-recognizes an ordered pair $(a,r^{\prime})$ such that $r^{\prime}\Vdash_{r}^{p}\psi(\text{decode}(a),\vec{p})$. 
\item If $\phi$ is of the form $\forall{x}\psi(x,\vec{p})$, then $r$ is an ordered pair $(P,\alpha)$ such that, whenever $a$ is a set and $c$ is a code for $a$, then $P^{c,r^{\prime}}(\alpha)$ c-recognizes a (code for a) set $r^{\prime}$ such that $r^{\prime}\Vdash_{r}^{p}\psi(a,\vec{p})$.
\item If $\phi$ is of the form $\forall{x\in X}\psi(x,\vec{p})$, then $r\Vdash_{r}^{p}\forall{x}(x\in X\rightarrow \psi(x,\vec{p}))$; if $\phi$ is of the form $\exists{x\in X}\psi(x,\vec{p})$, then $r\Vdash_{r}^{p}\exists{x}(x\in X\wedge\psi(x,\vec{p}))$. 
\item If $\phi(x_{1},...,x_{n})$ contains free variables occuring among $x_{1},...,x_{n}$, then an $r$-realizer for $\phi$ is defined to be an $r$-realizer for $\forall{x_{1},...,x_{n}}\phi$.
\end{enumerate}


If, instead of single ordinals, arbitrary sets of ordinals are allowed as parameters in clauses (5)-(7), then we obtain \textit{strong} $r$-realizability, denoted $r\Vdash_{r}^{s}\phi$.\footnote{This is the most direct variant of the realizability notion defined in \cite{CGP} using recognizability instead of computability.}

We say that $\phi$ is $r$-\textit{realizable} with parameters if and only if there is some set $r$ such that $r\Vdash_{r}^{p}\phi$; and similarly without parameters.

If we replace $c$-recognizability with OTM-computability (without parameters, with ordinal parameters, with set parameters) in this definition, we obtain OTM-realizability, the first two versions of were defined in \cite{CarlNote}, while the third was defined and studied in\cite{CGP}); we will denote this as $\Vdash_{\text{OTM}}$, decorated with no exponent, exponent $p$ and exponente $s$, depending on which version we refer to.

\end{defini}

\begin{remark}
In this paper, we will focus mainly on $r$-realizability without parameters and with ordinal parameters; the strong version will for the most part only be mentioned when results about it are covered by the same proofs that work for these variants. Indeed, many of  our results below hold for the version of $r$-realizability with ordinal parameters and with set parameters alike; we will then usually abuse our notation and just speak of $r$-realizability, which is understood to comprise both variants.
\end{remark}

The following lemma is the variant of the ``truth lemma'', see section $3$ of \cite{CGP}, for $r$-realizability; the proof follows the same strategy.

\begin{lemma}{\label{truth lemma}}
\begin{enumerate}
\item There is an OTM-program that, for any true $\Delta_{0}$-formula $\phi$, computes an $r$-realizer of $\phi$. In particular, a $\Delta_{0}$-formula is OTM-realizable if and only if it is $r$-realizable without parameters if and only if it is $r$-realizable with parameters if and only if it is strongly $r$-realizable if and only if it is true.
\item Let $\phi$ be a $\Pi_{1}$-formula. Then $\phi$ is $r$-realizable if and only if it is $r$-realizable with parameters if and only if it is strongly $r$-realizable if and only if it is true.
\item Let $\phi$ be a $\Sigma_{2}$-formula. Then $\phi$ is strongly $r$-realizable if and only if it is true.
\item Let $\phi$ be a $\Pi_{2}$-formula. Then if $\phi$ is (strongly) $r$-realizable, it is true.
\item Let $\phi$ be a $\Pi_{2}$-formula. Then if $\phi$ is $r$-realizable, it is strongly $r$-realizable.
\end{enumerate}
\end{lemma}
\begin{proof}
\begin{enumerate}
\item (Also cf. \cite{CGP2}.) An easy induction on formulas. 
We describe the functioning of the desired program $P_{\text{bt}}$ (``bounded truth''). For atomic formulas, $P_{\text{bt}}$ first decides whether the formula is true, and if it is, it outputs $\emptyset$; otherwise, it diverges.
The cases of conjunction and disjunction are trivial, and negation is reducible to implication. We now sketch how $P_{\text{bt}}$ works for implications and quantified statements.

Consider $\phi\rightarrow\psi$. Let an $r$-realizer $r$ for $\phi$ be given. By the inductive assumption, this means that $\phi$ is true, and since $\phi\rightarrow\psi$ is true, $\psi$ must be true. Thus, applying $P_{\text{bt}}$ to $\psi$ leads the desired result.

Consider $\exists{x\in X}\psi(x)$. Realizing this (for any of the stated variants) will mean realizing $\exists{x}(x\in X\wedge\psi(x))$. This, in turns, means computing/recognizing some $a$ and then realizing $a\in X\wedge\psi(a)$. Since $a\in X$ is quantifier-free, this is realized if and only if it is true. For $\psi(a)$, we can use the inductive hypothesis. Thus, realizability -- in any variant -- implies truth. On the other hand, if $\exists{x\in X}\psi(x)$ is true, it is realized by the program that, in the parameter $X$, searches through $X$, retrieves the first $a$ for which $\psi(a)$ is true (which can be identified since truth for $\Delta_{0}$-formulas is OTM-decidable by a result of Koepke \cite{Koepke}) and then, inductively computes a realizer for $\psi(a)$. 

Now consider $\forall{x\in X}\psi(x)$. If this statement is true, then, by inductive assumption, for any given $a$, $P_{\text{bt}}$ will compute a realizer for $\psi(a)$. On the other hand, if there is a realizer $(P,\alpha)$ for this statement, then, for any set $a$, we are guaranteed the existence of a realizer for $\psi(a)$, which implies that $\psi(a)$ is true.

\item Let $\phi$ be $\forall{x}\psi$, where $\psi$ is $\Delta_{0}$. 

If $\phi$ is true then, for every $a$, $\psi(a)$ is true and thus $\emptyset\Vdash_{r}\psi(a)$. Let $P_{\emptyset}$ be the program that, on any input, recognizes the empty set. Then $(P,\emptyset)\Vdash\phi$, and similarly with ordinals or sets of ordinals as parameters.

On the other hand, suppose that $(P,a)\Vdash_{r}^{s}\phi$. Then, for any $x$, $P^{a,x}$ $c$-recognizes a realizer $r_{x}$ for $\psi(x)$. In particular, there is such a realizer, which means that $\psi(x)$ is true. Thus, we have $\forall{x}\psi(x)$, i.e., we have $\phi$. Again, the argument is the same for the other two variants of $r$-recognizability. 

\item Let $\phi\equiv\exists{x}\forall{y}\psi(x,y)$, where $\psi$ is $\Delta_{0}$. 

Assume that $\phi$ is true. Pick $a$ for which we have $\forall{y}\psi(a,y)$, and let $c$ be a code for $a$.  Let $Q$ be the program that, in the oracle $x$, recognizes the pair $(x,P_{\emptyset})$. Then $(Q,c)\Vdash_{r}^{s}\phi$. 

On the other hand, if $(Q,a)\Vdash_{r}^{s}\phi$, then, for all $x$, $Q^{a,x}$ recognizes a strong $r$-realizer for $\forall{y}\psi(a,y)$; by (1), this implies that $\forall{y}\psi(a,y)$ is true, and hence, we have $\exists{x}\forall{y}\psi(x,y)$, i.e., $\phi$.

\item Let the statement $\phi$ in question be $\forall{x}\exists{y}\psi$, where $\psi$ is $\Delta_{0}$. A (strong) $r$-realizer of this statement consists of an OTM-program $P$ and a parameter $\rho$ such that, relative to any code $c$ for a set $x$, $(P,\rho)$ $c$-recognizes some $y$ with $\psi(x,y)$. This, in particular, implies the existence of such $y$, so that $\phi$ is true. 
\item It is easy to see that every $r$-realizer of a $\Pi_{2}$-formula $\phi$ is in fact a strong $r$-realizer of $\phi$. 
\end{enumerate}
\end{proof}

We remark that the converse of the second-last statement is in general not true, and that the last statement does not generalize to arbitrary formulas:

\begin{prop}
    Assume that $0^{\sharp}$ exists. Then there is a formula $\phi$ which is $r$-realizable, but not strongly $r$-realizable.
\end{prop}
\begin{proof}
    Consider the statement ``If $0^{\sharp}$ exists, then $1=0$''. Since ``$0^{\sharp}$ exists'' is not $r$-realizable, this statement is $r$-realized by $\emptyset$. On the other hand, since the existence of $0^{\sharp}$ is strongly $r$-realizable, the implication cannot be.
\end{proof}

\begin{prop}
It is consistent with ZFC that there is a true $\Pi_{2}$-statement $\phi$ which is not $r$-realizable (with or without (set) parameters).




\end{prop}
\begin{proof}
Consider the statement 
$$\phi:\Leftrightarrow\forall{x}\exists{y}(\neg x=\emptyset\rightarrow y\in x).$$

$\phi$ is obviously $\Pi_{2}$. Note that, by definition, we have $r\Vdash_{r}\neg x=\emptyset$ if and only if $x\neq\emptyset$: for in this case, $x=\emptyset$ has no realizers, making the definition of a realizer for implication vacuously true. 

Suppose that, in some transitive model $M$ of ZFC, there is $r\in M$ such that $r\Vdash_{r}^{s}\phi$; thus, $r=(P,\vec{p})$ for some program $P$ and some $\vec{p}\subseteq\text{On}$. Then, in $M$, we can define the class function $F:M\setminus\{\emptyset\}\rightarrow M$ by letting $F(X)=Y$ for $X\in M\setminus\{\emptyset\}$ if and only if 
$$M\models\exists{Q,\vec{q}\subseteq\text{On}}(P^{Q,\vec{q}}(\vec{p})\downarrow=1\wedge Q^{y}(\vec{q})\downarrow=1).$$ 
Then we must have $F(X)\in X$ for all $X\in M\setminus\{\emptyset\}$, which means that $F$ is a global choice function for $M$. Taking $M$ to be a model of ZFC where there is no global choice function (see, e.g. \cite{Levy}, p. 175f; the existence of such a function is equivalen to $V=$HOD), we see that $\phi$ fails to be $r$-realizable in $M$ and is thus independent of ZFC.

\end{proof}

\begin{remark}
The difference between realizability and truth (in the classical understanding of the term) is a standard feature of realizability notions that is, e.g., also well-known for Kleene realizability.
\end{remark}

We note a few important instances of Lemma \ref{truth lemma}.

\begin{corollary}{\label{truth lemma cases}}
    \begin{enumerate}
        \item For an ordered set $(x,<)$, the statement that $(x,<)$ is well-founded/a well-ordering is true if and only if it is $r$-realizable.
        \item If $x\in\text{On}$, then ``$x$ is a cardinal'' is $r$-realizable if and only if $x$ is a cardinal.
        \item For sets $x$ and $y$, the statement $y=\mathfrak{P}(x)$ is $r$-realizable if and only if $y=\mathfrak{P}(x)$.
    \end{enumerate}
\end{corollary}
\begin{proof}
    \begin{enumerate}
        \item Both statements are $\Pi_{1}$.
        \item The claim that $x$ is a cardinal can be expressed as ``For all $\alpha<x$ and all $f:\alpha\rightarrow x$, $f$ is not surjective''. As this is a $\Pi_{1}$-statement, the claim follows from Lemma \ref{truth lemma}.
        \item The claim that $y=\mathfrak{P}(x)$ can be expressed as $\forall{z}((z\in y\rightarrow z\subseteq x)\wedge(z\subseteq x\rightarrow z\in y))$. This is again $\Pi_{1}$, so we can apply Lemma \ref{truth lemma}. 
    \end{enumerate}
\end{proof}

If the set-theoretic universe is small, then $r$-realizability coincides with OTM-realizability:

\begin{lemma}{\label{v=l trivial}}
Let $\phi$ be an $\in$-statement. Then:
\begin{enumerate}
\item Suppose that $V=L$. Then $\Vdash_{\text{OTM}}\phi$ if and only if $\Vdash_{r}\phi$. 
\item Suppose that $V=L$. Then $\Vdash_{\text{OTM}}^{p}\phi$ if and only if $\Vdash_{r}^{p}\phi$.
\item Suppose that $V=L[x]$ for some $x\subseteq\text{On}$. Then $\Vdash^{s}_{\text{OTM}}\phi$ if and only if $\Vdash_{r}^{s}\phi$. 
\end{enumerate}
In fact, there is in both cases an OTM-effective method for computing realizers from $r$-realizers and vice versa.
\end{lemma}
\begin{proof}

Both proofs will work very similarly, and so we only elaborate on (2). Let $x$ be a set of ordinals such that $V=L[x]$, and let $P_{\text{enum}}$ be an OTM-program that, in the oracle $y$, successively writes codes for all elements of $L[y]$ to the output tape (which, of course, means that it never halts).\footnote{The existence of such a program is folklore, see, e.g., \cite{CSW}.} The point is now that, using $P_{\text{enum}}$, recognizable elements of $L[x]$ become computable by enumerating $L[x]$, applying the recognizing program to all outputs, and halting once a positive answer occurs; cf., e.g., \cite{LoMe}. 

The proofs will work by induction on $\phi$. For $\Delta_{0}$-formulas, all six notions agree by definition. If $\phi$ is $\psi_{0}\vee\psi_{1}$ or $\psi_{0}\wedge\psi_{1}$, the inductive steps are trivial. 

Let us first consider the implication from right to left: 

Let $\phi$ is $\psi_{0}\rightarrow\psi_{1}$, let $P$ be a program and $a\subseteq\text{On}$ be such that $(P,a)\Vdash_{r}^{s}\phi$. Thus, on input $r$ such that $r\Vdash_{r}^{p}\psi_{0}$, $P^{a,r}$ $c$-recognizes some $r^{\prime}$ with $r^{\prime}\Vdash_{r}^{s}\psi_{1}$. An OTM-realizer works as follows: On input (a code for) $r$, run $P_{\text{enum}}^{x}$. For each output $d$, run $P^{a,r}(d)$. If the output is $0$, continue running $P_{\text{enum}}$; if it is $1$, write $(d)_{0}$ to the output tape and halt. Since a realizer $r^{\prime}$ for $\psi_{1}$ exists in $L[x]$ by assumption, this is guaranteed to halt.

If $\phi$ is $\neg\psi$, recall that this is interpreted as $\psi\rightarrow 1=0$ and apply the last case.

If $\phi$ is $\forall{x}\psi(x)$, the argument is analogous to the implication case. 

Finally, let $\phi$ is $\exists{x}\psi(x)$. Suppose that $P$ is a program and $a\subseteq\text{On}$ is such that $(P,a)\Vdash_{r}^{s}\phi$. This means that $P^{a}$ recognizes a realizer $r$ for $\psi(\text{decode}(a))$. 
 Again, run $P_{\text{enum}}^{a}$ and, for every output $d$, run $P^{a,d}$ until the output is $1$, then write $(d)_{0}$ to the output tape and halt. As in the implication case, this is guaranteed to halt.

The procedure described is clearly implementable on an OTM, proving the last claim.

\bigskip

Now let us consider the implication from left to right. Again, we use induction on formulas. Suppose that $V=L[x]$ for some $x\subseteq\text{On}$ and that $\Vdash^{s}\phi$. Let $r$ be a (strong) OTM-realizer for $\phi$.

Suppose that $\phi$ is $\psi_{0}\rightarrow\psi_{1}$. Thus, relative to any OTM-realizer $r_{0}\Vdash^{s}\psi_{0}$, $r$ computes an OTM-realizer for $\psi_{1}$. To $r$-realize $\phi$, use $r$ (for a given $r_{0}\Vdash^{s}_{r}\psi_{0}$), to compute an OTM-realizer $r_{1}$ for $\psi_{1}$, and then compare it to the oracle. 

Negation and the universal quantification again work in analogy with the implication case.

If $\phi$ is $\exists{x}\psi$ and $r=(a,r^{\prime})\Vdash^{s}\phi$, then $r^{\prime}\Vdash\psi(\text{decode}(a))$. By induction, there is an $r$-realizer $s\Vdash_{r}^{s}\psi(\text{decode}(a))$. Now $(s,a)\Vdash_{r}^{s}\phi$.

Again, the transition from the OTM-realizer to an $r$-realizer is clearly effective.

\end{proof}

An important method for proving negative results on OTM-reducibility (\cite{Ca2018}, Lemma $7$) and OTM-realizability (\cite{CGP}) is the adaptation of Hodge's ``cardinality method'', which he used in \cite{Hodges} to show that certain field constructions are not ``effective'' in a certain sense. This method, along with the proof, also adapts to the present setting:

\begin{lemma}{\label{cardinality bound}} 
Let $x,y\subseteq\text{On}$. 
  Suppose that $x$ is recognizable relative to $y$ in the parameter $\rho\subseteq\alpha\in\text{On}$. Then $\text{card}(\text{tc}(x))\leq\text{card}(\text{tc}(y))+\aleph_{0}+\text{card}(\alpha)$. The same is true if $x$ is $c$-recognizable relative to $y$.
\end{lemma}
\begin{proof}
Let $P$ be a program that recognizes $x$ relative to $y$ in the parameter $\rho$. 
Pick an infinite cardinal $\kappa$ large enough so that $\alpha$, $\text{tc}(y)$, $\text{tc}(x)$ belong to H$_{\kappa}$ (the set of sets heriditarily of cardinality $<\kappa$). 
Then all ordinal encodings of $x$ and $y$ will belong to H$_{\kappa}$ as well. 
Form the elementary hull $H$ of $\{\rho\}\cup\alpha+1\cup\{y\}\cup\text{tc}(y)$ in H$_{\kappa}$. 
Clearly, we have $\text{card}(H)\leq\text{card}(\text{tc}(y))+\text{card}(y)+\text{card}(\alpha) \leq \text{card}(y)+\aleph_{0}+\text{card}(\alpha)$. 
Moreover, since $H_{\kappa}\models\exists{u}P^{u,y}(\rho)\downarrow=1$, the same is true of $H$, and since $x$ is unique with this property, we have $x\in H$. 
Form the transitive collapse $\overline{H}$ of $H$, and let $\pi:H\rightarrow\overline{H}$ be the collapsing map. 
Since $\rho\subseteq\alpha+1\subseteq H$ and $\text{tc}(y)\subseteq H$, we have $\pi(\rho)=\rho$ and $\pi(y)=y$. 
By elementarity, we have $\overline{H}\models P^{\pi(x),\pi(y)}(\rho)\downarrow=1$, i.e., $\overline{H}\models P^{\pi(x),y}(\rho)\downarrow=1$. 
By absoluteness of computations, we have $P^{\pi(x),y}(\rho)\downarrow=1$. 
Since $P$ recognizes $x$ in $y$ and $\rho$, we must have $\pi(x)=x$. Thus $x\in\overline{H}$, and as $\overline{H}$ is transitive and by elementarity, we have $\text{tc}(x)\in H$. But now, we have $\text{card}(\text{tc}(x))=\text{card}(\pi(\text{tc}(x))))\leq\text{card}(\overline{H})\leq\text{card}(\text{tc}(y))+\aleph_{0}+\text{card}(\alpha)$, as desired.

Now, if $x$ is $c$-recognizable relative to $y$, pick $z$ such that $z$ is recognizable relative to $y$ with $(z)_{0}=x$. Clearly, this implies $\text{card}(x)\leq\text{card}(z)$, so the required statement follows from the last paragraph. 
\end{proof}

\begin{lemma}{\label{cardinality method}} [Cf. Hodges \cite{Hodges}, S. 144f. and Lemma 3.2]
Suppose that $\phi(x,y)$ is such that, for every cardinal $\kappa$, there is some $x$ with $\text{card}(x)\geq\kappa$ such that, for every $y$ with $\Vdash_{r}\phi(x,y)$, we have $\text{card}(\text{tc}(y))>\text{card}(\text{tc}(x))$. Then $\forall{x}\exists{y}\phi(x,y)$ is not (strongly) $r$-realizable.
\end{lemma}
\begin{proof}
Suppose for a contradiction that $(P,\alpha)\Vdash_{r}\forall{x}\exists{y}\phi(x,y)$. By the assumption of the statement, pick $x$ such that $\text{card}(x)>\alpha+\aleph_{0}$ and such that, for every $y$ with $\phi(x,y)$, we have $\text{card}(y)>\text{card}(x)$. By the choice of $(P,\alpha)$, there is some $y$ such that $\phi(x,y)$ and some code for $y$ is $c$-recognizable relative to each code for $x$. By Lemma \ref{cardinality bound}, we $\text{card}(x)<\text{card}(y)\leq\text{card}(x)+\aleph_{0}+\text{card}(\alpha)=\text{card}(x)$, a contradiction.
\end{proof}

\begin{corollary}{\label{Pot and cards}}
    \begin{enumerate}
        \item The power set axiom Pot is neither $r$-realizable nor strongly $r$-realizable.
        \item The statement that, for every ordinal, there is a larger cardinal is neither $r$-realizable nor strongly $r$-realizable.
    \end{enumerate}
\end{corollary}
\begin{proof}
An easy condensation argument shows that, if $x$ is $c$-recognizable relative to $y$, then the cardinality of $x$ cannot be larger than the cardinality of $y$. This immediately yields both claims.
\end{proof}

Since comprehension and replacement are schemes, rather than single formulas, it needs to be clarified what it means to realize these axioms. As noted in \cite{CGP} (p. 18f), there are two natural interpretations: In the uniform interpretation, there is a single pair $(P,\rho)$ of a program and a parameter such that, for any code $c$ for a set $x$, and parameter $\rho^{\prime}$ and any formula index $n$ (in some natural enumeration $(\phi_{i}:i\in\omega)$ of the $\in$-formulas), $P^{c}(n)$ $c$-recognizes a code for $\{z\in x:\phi(z,\rho^{\prime})\}$; the non-uniform interpretation only requires the existence of such a pair $(P_{\phi,\rho^{\prime}},\rho_{\phi,\rho^{\prime}})$ for each choice of a formula and a parameter separately. Clearly, the uniform variant implies the non-uniform variant. We will now show that even the non-uniform variant fails for comprehension, while the uniform variant is true for replacement.

\begin{lemma}{\label{comprehension}}
\begin{enumerate}
\item The separation scheme is not $r$-realizable. In fact, it already fails for $\Pi_{1}$-formulas (and thus also for $\Sigma_{1}$-formulas) to be $r$-realizable.
\item The separation scheme is not strongly $r$-realizable. In fact, it already fails for $\Pi_{1}$-formulas (and thus also for $\Sigma_{1}$-formula) to be $r$-realizable.
\end{enumerate}
\end{lemma}
\begin{proof}
Let $\phi(x)$ be the statement ``$x$ is a cardinal'' (which is expressible as a $\Pi_{1}$-formula stating that every set is not a bijection between $x$ and an element of $x$). Let $C_{\phi}$ be the corresponding instance of the separation scheme, i.e., the statement $\forall{x}\exists{y}((\forall{z\in y}(z\in x\wedge\phi(z))) \wedge \forall{z\in x}(\phi(z)\rightarrow z\in y))$; let us abbreviate the bracketed part by $\psi(x,y)$. Note that $\psi$ is $\Pi_{1}$. We claim that is neither $r$-realizable nor strongly $r$-realizable. Suppose for a contradiction that $r\Vdash_{r}C_{\phi}$. That is, $r$ consists of a pair $(P,\rho)$ of an OTM-program $P$ and a parameter $\rho$ such that, relative to any code $c_x$ for a set $x$, $P^{c_{x}}(\rho)$ $c$-recognizes some $s$ such that $s\Vdash_{r}\exists{y}\psi(x,y)$. This $s$, in turn, will, relative to $c_x$, $c$-recognize a set $y$, along with some $t$ such that $t\Vdash_{r}\psi(x,y)$. Combining the two yields a program which, relative to any given $c_x$, $c$-recognizes a pair $(y,t)$ with $t\Vdash_{r}\psi(x,y)$. Recall from Corollary \ref{truth lemma cases} that $\phi(x)$ being $r$-realizable is equivalent to its being true, i.e., $x$ being a cardinal. Now, if $t\Vdash\psi(x,y)$, it follows that $t=(t_{0},t_{1})$ such that $t_{0}\Vdash_{r}(\forall{z\in y}(z\in x\wedge\phi(z)))$ and $t_{1}\Vdash_{r}\forall{z\in x}(\phi(z)\rightarrow z\in y)$. Further unfolding 
the definitions implies that both parts, and hence $\psi(x,y)$, are actually true; that is, $y$ is actually the set of cardinals in $x$. Taken together, we have obtained that (some code for) the set of cardinals in $x$ is $c$-recognizable relative to $c_x$, uniformly in $c_x$. Let $(Q,\gamma)$ be an OTM-program and a parameter that achieve this. We can thus express the claim that $\alpha$ is a cardinal as $\exists{z}(Q^{z}(\gamma)\downarrow=1\wedge \alpha\in(z)_{0})$, which is a $\Sigma_{1}$-formula. It is, however, well-known that being a cardinal is not expressible as a $\Sigma_{1}$-formula. 
 (To recall the argument, assume for a contradiction that it is, pick an ordinal $\alpha$ sufficiently large so that $\gamma\subseteq\alpha$, and consider the first $\beta>\alpha$ so that $L_{\beta}\models$``There is a largest cardinal $\nu$''. Then, by upwards absoluteness of $\Sigma_1$-statements, $\nu$ would be a cardinal, but it is not hard to see that it is not even a cardinal in $L_{\nu+1}$.)

For $\Sigma_{1}$-formulas, just note that if the set $s$ of non-cardinals in $x$ -- which is $\Sigma_{1}$-definable -- was $c$-recognizable, then, being OTM-computable from $s$ and $x$, so was its complement, contradicting what we just showed.

\end{proof}


\begin{lemma}{\label{replacement realizable}}
The replacement scheme is uniformly $r$-realizable.
\end{lemma}
\begin{proof}

Let $\phi$ be an $\in$-formula, $\vec{p}\subseteq\text{On}$ a parameter, $X$ a set coded by $c_{X}\subseteq\text{On}$. Our goal is to, uniformly in $\phi$, $\vec{p}$ and $c_{X}$, $r$-realize the statement $$\forall{x\in X}\exists{y}\phi(x,y)\rightarrow\exists{Y}\forall{x\in X}\exists{y\in Y}\phi(x,y).$$

        Consider some $r$ such that $r\Vdash_{r}^{p}\forall{x\in X}\exists{y}\phi(x,y)$. Hence, $r$ is a pair $(P,\rho)$, where $P$ is an OTM-program and $\rho$ is a parameter such that, for each code $c_{x}$ for some $x\in X$, $P^{c_{x}}(\rho)$ $c$-recognizes a pair $(y(c_{x}),r_{x})$ such that $r_{x}\Vdash_{r}\phi(x,\text{decode}(y(c_{x})))$. 
        Let $Y:=\{y(c_{x}^{c_{X}}):x\in X\}$. We claim that a code for $Y$ is uniformly $c$-recognizable from $\phi$, $\vec{p}$ and $c_{X}$. 
        To this end, let $W:=\{(c_{x}^{c_{X}},c_y,r):x\in X\wedge P^{(c_{x}^{c_{X}},(c_y,r))}(\rho)\downarrow=1\}$. 
        We claim that (some code for) $(Y,W)$ is uniformly recognizable in $\phi$, $\vec{p}$ and $c_{X}$.

        So let a code for a set $Z$ be given in the oracle. We start by checking whether $Z$ codes a set of the form $(A,B)$, where $B$ is a set of triples and $A$ is the set of second components of elements of $B$; if not, we halt with output $0$. Otherwise, it remains to check whether $B=W$. This will be done in two steps.

        In the first step, we run through $X$. For each $x\in X$, we check whether there is precisely one element $(x,a,b)$ with first component $x$ in $W$. If not, we halt with output $0$. Otherwise, we run $P^{(x,(a,b))}(\rho)$. If the output is $0$, we halt with output $0$. Otherwise, we continue with our run through $X$. If this step finishes after considering all elements of $X$, we are guaranteed that $Y\subseteq A$. 

        In the second step, we run through $A$. For each $a\in A$, we check (by running through $W$) whether $W$ contains a triple of the form $(x,(y,r))$ such that $P^{(x,(y,r))}(\rho)\downarrow=1$. If that is the case, then $A\subseteq Y$, and we halt with output $1$.

\end{proof}

\begin{lemma}{\label{ac realizable}}
Let AC be the statement that, for every $x$, if every $y\in X$ has an element, then there is a function $f:X\rightarrow\bigcup{X}$ such that, for all $y\in x$, we have $f(y)\in y$.
Then AC is $r$-realizable.
\end{lemma}
\begin{proof}
We describe an $r$-realizer $r_{\text{AC}}$ for AC. Since AC is an implication, this will consist of an OTM-program -- and, potentially, parameters -- that will compute a realizer for ``there is a choice function for $X$'' whenever it is given an $r$-realizer of ``every $y\in X$ has an element''. 

So pick an $r$-realizer $r$ of the latter statement, and let $c_{X}$ be a code for a set of non-empty sets. $r$ will be a pair $(P,\rho)$ such that, for every code $c_y$ of an element $y\in X$, $P^{c_{y}}(\rho)$ recognizes a code $c_{x}(c_{y})$ for an element $x$ of $y$. Let $f:=\{(c_{y}^{c_{X}},c_{x}(c_{y}^{c_{X}})):y\in X\}$. Clearly, $f$ encodes a choice function for $X$. We claim that (some code for) $f$ is recognizable relative to $c_{X}$. The statement ``$f$ is a choice function for $x$'' is then a true $\Delta_{0}$-statement, which is trivially $r$-realizable. 

We describe how to recognize $f$, given (a code for) $x$, $P$, and $\rho$. Let $a\subseteq\text{On}$ be given in the oracle. First check whether $a$ codes a choice function for $x$; this is possible because it only requires evaluating bounded formulas. Now run through $X$ (using $c_{X}$). For every $y\in X$, compute $c_{y}^{c_{X}}$, then a code $d$ for $a(y)$ and run $P^{c_{y},d}(\rho)$. If the answer is $0$, halt with output $0$. If the answer is $1$, continue with the next element of $x$. Once all elements of $x$ have been considered, halt with output $0$.
\end{proof}

\subsection{Kripke-Platek Set Theory is $r$-realizable}

In this section, we will show that the axioms of Kripke-Platek set theory are $r$-realizable. For (computable)   OTM-realizability, this was demonstrated in \cite{CGP}; most of the arguments are rather straightforward, and resemble those given in \cite{CGP}.

\begin{lemma}{\label{trivial axioms}}
The following KP-axioms are (strongly) $r$-realizable, with and without parameters:
    \begin{enumerate}
        \item Extensionality, i.e., $\forall{x,y}(x=y\leftrightarrow\forall{z}(z\in x\leftrightarrow z\in y))$.
        \item Pairing.
        \item Existence of the empty set.
        \item The axiom of union.
        \item Existence of an infinite set.    
    \end{enumerate}
\end{lemma}
\begin{proof}
For extensionality, let sets $x,y$ be given. If an $r$-realizer for $x=y$ is given, then, by definition, $x=y$ is true, and hence so is $\forall{z}(z\in x\leftrightarrow z\in y)$, which is $r$-realized by the a program that, regardless of the oracle, recognizes $\emptyset$. On the other hand, if an $r$-realizer for $\forall{z}(z\in x\leftrightarrow z\in y)$ is given, then $\forall{z}(z\in x\leftrightarrow z\in y)$ is true by Lemma \ref{truth lemma}, so $x=y$ is realized by any set, and we can again use a program that recognizes the empty set in any oracle.

For (2)-(5), note that it is easily seen that there are parameter-free OTM-programs for computing (i) a code for $\{x,y\}$ from codes for $x$ and $y$, (ii) a code for $\emptyset$, (iii) a code for $\bigcup{X}$ from a code for $X$, (iv) a code for $\omega$. Since computability implies recognizability, we obtain the desired results.

\end{proof}

\begin{lemma}{\label{separation}} 
    The $\Delta_{0}$-separation scheme is uniformly (strongly) $r$-realizable.
\end{lemma}
\begin{proof} 
Let $\phi$ be a $\Delta_{0}$-formula,  $\vec{p}\subseteq\text{On}$ be a parameter, $X$ a set coded by $c_{X}\subseteq\text{On}$. 
        We need to show how to, uniformly in $\phi$, $\vec{p}$ and $X$, recognize a code for the set $Y:=\{x\in X:\phi(x,\vec{p}\}$. But it is easy to see that (using the ability of OTMs to evaluate bounded truth predicates, see \cite{Koepke}) a code $c$ for $Y$ is OTM-computable from the given data. Thus, $Y$ is uniformly recognizable in $\phi$, $\vec{p}$ and $X$.  
\end{proof}

\begin{lemma}{\label{kp induction r-realizable}}
The axiom of $\in$-induction, i.e., the statement 

\begin{center}
    $\text{Ind}(\phi):\Leftrightarrow \forall{p}(\forall{a}(\forall{x\in a}\phi(x,p)\rightarrow\phi(a,p))\rightarrow\forall{y}\phi(y,p))$
\end{center}

is $r$-realizable for all formulas $\phi$ (with and without parameters), and in fact uniformly in $\phi$.
\end{lemma}
\begin{proof}
The proof is standard and closely follows the example of the one used in \cite{CGP} for strong OTM-realizability (\cite{CGP}, Theorem $40$)). The idea there is the natural one, namely to build up the required realizers by $\in$-recursion on the elements of the transitive closure of the given element $y$. We will follow the same route; however, in the present context, an extra level of complication arises from the fact that the occuring realizers can themselves not be computed, but only recognized, relative to other realizers for which the same is true, all the way down to the empty set. We thus give the full details.

We show how to compute (and thus recognize) from any given formula $\phi$ (encoded in some natural way as a natural number) and any parameter $p$ an $r$-realizer $r_{\phi}$ for $\text{Ind}(\phi)$. 

In other words, relative to $\phi$, $p$, an $r$-realizer $r\Vdash_{r}\forall{a}(\forall{x\in a}\phi(x,p)\rightarrow\phi(a,p))$, and a (code for a) set $y$, we need to $c$-recognize an $r$-realizer $r_{y}\Vdash_{r}\phi(y,p)$. 

From the given code $c_y$ for $y$, it is easy to compute the transitive closure $\text{tc}(y)$ of $y$.\footnote{This, in fact, is trivial by the encoding we use: If the code $c_y$ is derived from the bijection $f:\alpha\rightarrow\text{tc}(\{y\})$, then $c_{y}\cup\{p(\xi,\alpha+1):\xi\in\alpha\}$ is a code for $\text{tc}(y)$.} Let us assume that the elements of $\text{tc}(y)$ are written on a separate tape. Given $c_{y}$, codes for elements of $\text{tc}(y)$ can be derived from it; for the sake of simplicity, we will confuse $z$ and $c_{z}^{c_{y}}$ below. 

We will now recursively associate with each $z\in y$ two realizers $c_{<z}\Vdash_{r}\forall{x\in z}\phi(x,p)$ and $c_{z}\Vdash_{r}\phi(z,p)$. More precisely, we will construct a program $r^{\prime}$ such that, for all $z\in y$, $(r^{\prime})^{z}(0)$ $c$-recognizes $c_{<z}$ while $(r^{\prime})^{z}(1)$ $c$-recognizes $c_{z}$. (Obviously, such a program can be obtained by combining two separate programs for each of these tasks into one.) 
Note that $c_{z}$ is $c$-recognizable from $c_{<z}$, as $c_{z}=\rho_{0}(\rho_{0}(r,z),c_{<z})$: Relative to $z$, $r$ $c$-recognizes an $r$-realizer for $\forall{x\in z}\phi(x,p)\rightarrow\phi(z,p)$, and $c_{<z}$ is an $r$-realizer for $\forall{x\in z}\phi(x,p)$.

We will use the following notation:  Let us write $e_{z}:=\rho(\rho_{0}(r,z),c_{<z})$, for all $z\in\text{tc}(\{y\})$. (Thus, $e_{z}$ is the $r$-realizer $c_{z}$ obtained via $r$ from $r_{<z}$, together with the extra information used in $c$-recognizing it.) 
For $z\in\text{tc}(\{y\})$, we let $\rho_{z}:=\{(x,\rho(r,x)):x\in\text{tc}(\{z\})$. It is easy to see that a code for $\rho_{z}$ is $r$-recognizable relative to $r$, the given encoding of $\text{tc}(\{y\})$, and $z$, for all $z$, by encoding the pairs by interleaving the codes for their components and then combining these to a code for the set of pairs in some canonical manner. This code can then be recognized by checking (i) whether the first components of the elements of the coded set form the set $\text{tc}(\{z\})$ and (ii) whether, for each such pair $(u,v)$, $v=\rho(r,u)$. We can -- and will -- regard $\rho_{y}$ as a ``lookup table'', which, for every $z\in\text{tc}(\{y\})$, will yield an $r$-realizer $\rho_{0}(r,z)\Vdash_{r}\forall{x\in z}\phi(x,p)\rightarrow\phi(z,p)$. Since $\rho_{y}$ is recognizable, we may -- and will -- assume it as given in the following.\footnote{Thus, when we say below that $T_{z}$ is $c$-recognizable relative to $z$, we really mean that $(T_{z},\rho_{y})$ is $c$-recognizable relative to $z$.}
Moreover, let us, for all $z\in\text{tc}(\{y\})$, write $T_{z}:=\{(x,(c_{<x},e_{x})):x\in\text{tc}(\{z\})\}$. 
It is again easy to encode $T_{z}$ 
as a set of ordinals in some natural way, and we will from now on assume that such an encoding has been fixed and slightly abuse notation by using $T_{z}$ 
to refer to this code. 

We will in fact show that $(T_{z},\rho_{z})$ 
is uniformly $c$-recognizable relative to $z$ and $r$. 
Since $(c_{<z},c_{z})$ is easily $r$-recognizable from $T_{z}$, 
this will establish the existence of the desired $r^{\prime}$. Since the recognizability of $\rho_{z}$ is already established, we only need to deal with $T_{z}$; we will show that it is $c$-recognizable relative to $z$ and $r$ by a program $s$.

$s$ will work by $\in$-recursion on the input $z$. 
Let some set $T$ be given in the oracle.  We start by checking whether $T$ encodes a function with domain $\text{tc}(z)$. If not, we halt with output $0$. Otherwise, we do the following for every $x\in\text{tc}(z)$: 

\begin{enumerate}
\item We compute (the encoding for) the restriction $T\upharpoonright\text{tc}(\{x\})$ of $T$ to $\text{tc}(\{x\})$.
\item We run $s^{x}(T\upharpoonright\text{tc}(\{x\})$. 
\item If the output is $0$, we halt with output $0$. Otherwise, we continue with the next $x$. 
\end{enumerate}

Let us denote this subroutine consisting of (1)-(3) by $s^{\prime}$ (which runs uniformly in the input $z$). If it finishes without having produced the output $0$ along the way, we know that, for all elements of $\text{tc}(y)$, $T$ gets the values right; it then only remains to check that it assigns the right value $(c_{<y},e_{y})$ to $y$ itself. So let $T(y)=:(u,v)$. Now, $u$ should just be the OTM-program that does the following: 

Given the input $d$ (the element of $\text{tc}(y)$ for which we want to obtain an $r$-realizer for $\phi(d,p)$) 
and the oracle $(T,t)$, it first runs $s^{\prime}$ on $T$ in the input $y$ to check whether $T=T_{y}\upharpoonright\text{tc}(y)$; if not, we halt with output $0$. If yes, the definition of $T_{y}$ implies that the desired $r$-realizers can now simply be read off from $T$: $T(d)$ will be $(c_{<d},e_{d})$, and $c_{d}$ -- which is by definition an $r$-realizer of $\phi(d,p)$ -- is the second component of $e_{d}$, so we can compute $((T(c))_{1})_{0}$ and check whether this is equal to $t$.

If $u$ is this program,\footnote{Note that this program is a specific natural number and will be the same for all $y$; its inclusion into the object to be recognized serves merely presentational purposes: We find it more natural to think about the routine in this way.} we have $u=c_{<y}$, and we can continue to check whether $v=e_{y}$. This works as follows: Recall that, by definition, $(\rho_{y}(y))\Vdash_{r}\forall{x\in y}\phi(x,p)\rightarrow\phi(y,p)$, and that $\rho_{y}$ is available because $\rho$ is recognizable. Hence, we can run $(\rho_{y}(y))^{c_{<y}}(v)$ to determine whether $v=e_{y}$. 

Now, $r^{\prime}$ work as follows: On input $y$ and in the oracle $(a,b,c)$, it will first check whether $c=\rho$; if not, it halts with output $0$. Otherwise, it uses $c$ to check whether $b=T_{y}$; if not, it halts with output $0$. Otherwise, it finally checks whether $a=((T_{y}(y))_{1})_{0}$; if not, it halts with output $0$, otherwise, it halts with output $1$.







\end{proof}

Taken together, the above Lemmas imply:

\begin{thm}{\label{kp r-realizable}}
    All axioms and axiom schemes of KP are $r$-realizable (with and without parameters).
\end{thm}

\section{$r$-realizability and intuitionistic provability}

In this section, we prove a version of the Curry-Howards-correspondence for $r$-realizability; that is, if all elements of a set $\Phi$ of $\in$-sentences are $r$-realizable and $\phi$ is an intuitionistic consequence of $\Phi$, then $\phi$ is also $r$-realizable. With strong OTM-realizability in place of $r$-realizability, this was done in \cite{CGP}.

We begin by noting that for any $\phi$, exactly one of $\phi$ and $\neg\phi$ is $r$-realizable. This is just the usual sense in which the realizability interpretation satisfies the law of excluded middle. We note, though, that this sense is rather weak: Parts (3) and (4) of the next proposition give some more concrete meaning to this ``weakness'' (cf. also \cite{CGP}). 


\begin{defini}
    For an $\in$-sentence $\phi$, let us call the ``atomic negation'' of $\phi$, written $\rightharpoondown\phi$ the formula $\psi$ that -- in classical logic -- arises from $\neg\phi$ by pushing the negation sign inwards until it only appears in front of atomic subformulas.
\end{defini}

\begin{prop}{\label{logical properties}}
\begin{enumerate}
\item For each $\in$-sentence $\phi$, one of $\phi$ and $\neg\phi$ is $r$-realizable.
\item For no $\in$-formula $\phi$ are $\phi$ and $\neg\phi$ both $r$-realizable.    
\item There is an $\in$-sentence $\phi$ such that neither $\phi$ nor $\rightharpoondown\phi$ are $r$-realizable.
\item There is an $\in$-formula $\phi(x)$ such that $\forall{x}(\phi(x)\vee\neg\phi(x))$ is not $r$-realizable.
\end{enumerate}
\end{prop}
\begin{proof}
\begin{enumerate}
\item If $\phi$ is not $r$-realizable, i.e., if there is no $r$-realizer for $\phi$, then, by definition, $\phi\rightarrow 1=0$ is $r$-realized by $\emptyset$.
\item     Otherwise, we would have $r$-realizers for $\phi$ and for $\phi\rightarrow1=0$; but then, we could recognize an $r$-realizer for $1=0$, which is quantifier-free and hence $\Delta_{0}$, so by Lemma \ref{truth lemma}, $1=0$ would be true, a contradiction.
\item 
Now, $\rightharpoondown\phi$ states that there exists an ordinal $\alpha$ such that no cardinal is greater than $\alpha$; that is, there is $\alpha$ such that, for all $\beta$, there is a surjection $f:\alpha\rightarrow\beta$. This statement being $r$-realizable would imply the existence of an ordinal $\alpha$ for which the statement ``for all $\beta$, there is a surjection $f:\alpha\rightarrow\beta$'' is $r$-realizable. This statement is $\Pi_{2}$, and hence its $r$-realizability implies its truth. But clearly, it is false.
\item Let $\phi(x)$ be the statement that $x$ is a cardinal. By Corollary \ref{truth lemma cases}(2), $\phi(x)$ is true if and only if $x$ actually is a cardinal. Now, if $\forall{x}(\phi(x)\vee\neg\phi(x))$ was $r$-realizable, say, by the program $P$ in the parameter $\rho$, then we would have $\phi(x)\Leftrightarrow\exists{y}P^{0\oplus y}\downarrow=1$. Thus, the property of being a cardinal would be $\Sigma_{1}$-expressable in some parameter, but as we recalled at the end of the proof of Lemma \ref{comprehension}, it is not.

\end{enumerate}
\end{proof}


We consider Hilbert calculus for intuitionistic provability as given in Aczel and Rathjen \cite{Aczel}, Definition 2.5.1. There, we find the a list of axioms along with three deduction rules. We will now argue that every instance of these axioms is $r$-realizable (for all notions of $r$-realizability considered here) and that $r$-realizability is preserved by the deduction rules.

\begin{lemma}{\label{propositional axioms}}
All propositional instances of the following propositional formulas (see \cite{Aczel}, Def. 2.5.1) are $r$-realizable; that is, when replacing each propositional variable by an $\in$-formula (equal variables by equal variables, of course), then the resulting formula is $r$-realizable. The proofs follow those in \cite{CGP}, with recognizability replacing computability. 

\begin{enumerate}
    \item $\phi\rightarrow(\psi\rightarrow\phi)$
    \item $(\phi\rightarrow(\psi\rightarrow\xi))\rightarrow((\phi\rightarrow\psi)\rightarrow(\phi\rightarrow\xi))$
    \item $\phi\rightarrow(\psi\rightarrow(\phi\wedge\psi))$
    \item $(\phi\wedge\psi)\rightarrow\phi$ and $(\phi\wedge\psi)\rightarrow\psi$
    \item $\phi\rightarrow(\phi\vee\psi)$ and $\psi\rightarrow(\phi\vee\psi)$
    \item $(\phi\vee\psi)\rightarrow((\phi\rightarrow\xi)\rightarrow((\psi\rightarrow\xi)\rightarrow\xi))$
    \item $(\phi\rightarrow\psi)\rightarrow((\phi\rightarrow\neg\psi)\rightarrow\neg\phi)$
    \item $\phi\rightarrow(\neg\phi\rightarrow\psi)$
\end{enumerate}
\end{lemma}
\begin{proof}
\begin{enumerate}
    \begin{enumerate}
        \item Let $P$ be the program which, on input $r$, recognizes the program that checks whether the input is equal to $r$. 
        \item Let $r$ be an $r$-realizer for $(\phi\rightarrow(\psi\rightarrow\xi))$. We show how to $c$-recognize, relative to $r$, an $r$-realizer for $(\phi\rightarrow\psi)\rightarrow(\phi\rightarrow\xi)$. So let an $r$-realizer $s\Vdash_{r}\phi\rightarrow\psi$ be given; we need to $c$-recognize a realizer for $\phi\rightarrow\xi$. Let $t\Vdash_{r}\phi$ be given. Relative to $t$, $r$ will $c$-recognize an $r$-realizer $t^{\prime}\Vdash_{r}(\psi\rightarrow\xi)$. Relative to $s$, $r$ will $c$-recognize an $r$-realizer $s^{\prime}\Vdash\phi\rightarrow\psi$. Relative to $s^{\prime}$, $t^{\prime}$ will $c$-recognize an $r$-realizer $t^{\prime\prime}\Vdash_{r}\xi$. By transitivity of $c$-recognizability, this works as desired.
More concisely, given $r$, $s$, and $t$, our realizer works by recognizing the triple $(\rho(\rho(\rho_{0}(r,t),\rho_{0}(r,s)),\rho(r,s),\rho(r,t))$, which can be done in the way described, proceeding from right to left.
        \item Let an $r$-realizer $r\Vdash_{r}\phi$ be given. We need to $c$-recognize, relative to $r$, an $r$-realizer for $\psi\rightarrow(\phi\wedge\psi)$. Now, relative to any given $s\Vdash_{r}\psi$, it is easy to recognize $(r,s)$, which $r$-realizes $(\phi\wedge\psi)$.
        \item On input $(r,s)$, one can $c$-recognize both $r$ and $s$.
        \item Let $P$ be the program which, on input $r$, recognizes $(0,r)$ or $(1,r)$, respectively.
        \item We need to $c$-recognize, relative to any given realizer $r\Vdash_{r}(\phi\vee\psi)$, an $r$-realizer for $((\phi\rightarrow\xi)\rightarrow((\psi\rightarrow\xi)\rightarrow\xi))$. Now, $r$ is either $(0,r^{\prime})$ with $r^{\prime}\Vdash_{r}\phi$ or $(1,r^{\prime})$ with $r^{\prime}\Vdash_{r}\psi$. 
        
        In the first case (i.e., if the first component of $r$ is $0$), we recognize the OTM-program $P$ which, for any given realizer $s\Vdash_{r}\phi\rightarrow\xi$, proceeds as follows: Relative to $r^{\prime}$, $s$ will $c$-recognize an $r$-realizer $s^{\prime}$ for $\xi$. To $r$-realize $(\psi\rightarrow\xi)\rightarrow\xi$, one can now ignore the given realizer for $\psi\rightarrow\xi$ and simply recognize $s^{\prime}$. Note that the code for $P$ is easily computable, and thus in particular recognize, from $r$. 

        Otherwise, if the first component of $r$ is $1$ -- so that $r^{\prime}\Vdash_{r}\psi$ -- $P$ will work as follows: We ignore the given $r$-realizer for $(\phi\rightarrow\xi)$; given an $r$-realizer $s\Vdash_{r}(\psi\rightarrow\xi)$, $s$ will recognize an $r$-realizer $t\Vdash_{r}\xi$ relative to $r^{\prime}$, which is what we want.
        \item Given an $r$-realizer $r\Vdash_{r}(\phi\rightarrow\psi)$, we compute -- and hence recognize -- relative to $r$ the OTM-program $P$ that proceeds as follows: Let $s\Vdash_{r}\phi\rightarrow\neg\psi$ be given. We need to $c$-recognize a realizer for $\neg\phi$, i.e., for $\phi\rightarrow 1=0$. This will work as follows: Given an $r$-realizer $t\Vdash_{r}\phi$, use $r$ to $c$-recognize, relative to $t$, an $r$-realizer $t^{\prime}\Vdash_{r}\psi$; moreover, use $s$ to $c$-recognize an $r$-realizer $s^{\prime}\Vdash_{r}\phi\rightarrow\neg\psi$. Now, $s^{\prime}$ will, relative to $t$, $c$-recognize an $r$-realizer $t^{\prime\prime}$ for $\neg\psi$, i.e., for $\psi\rightarrow 1=0$. Finally, $t^{\prime\prime}$ will, relative to $t^{\prime}$, $c$-recognize an $r$-realizer for $1=0$.
        \item Relative to any given $r\Vdash_{r}\phi$, we can recognize the empty set. Now, if $\phi$ is $r$-realizable, then $\neg\phi$ is not by Proposition \ref{logical properties}(2), so $\emptyset\Vdash_{r}\neg\phi\rightarrow\psi$.
    \end{enumerate}
\end{enumerate}
%
%
\end{proof}

\begin{lemma}{\label{first-oder axioms}}
    For an $\in$-formula $\phi$, a variable $x$ and a term $t$, we write $\phi[\frac{t}{x}]$ for the formula obtained from $\phi$ by replacing all free occurences of $x$ in $\phi$ by $t$.
    
    All first-order instances of the following first-order formulas (see \cite{Aczel}, Def. 2.5.1) are $r$-realizable; that is, when replacing each formula variable by an $\in$-formula using the same free variables, the resulting formula will be $r$-realizable.
    \begin{enumerate}
        \item $\forall{x}\phi\rightarrow\phi[\frac{t}{x}]$, where $t$ is free for $x$ in $\phi$
        \item $\phi[\frac{t}{x}]\rightarrow\exists{x}\phi$, where $t$ is free for $x$ in $\phi$
        \item $x=x$
        \item $x=y\rightarrow(\phi[\frac{s}{x}]\rightarrow\phi[\frac{t}{x}])$, where $s$, $t$ are free for $x$ in $\phi$
    \end{enumerate}
\end{lemma}
\begin{proof}  
Note that, in the $\in$-language, there are neither function nor constant symbols, so that the only terms are variables. 
\begin{enumerate}
    \item Given an $r$-realizer $r\Vdash_{r}\forall{x}\phi$, $r$ will also $r$-realize $\forall{y}\phi$ for any variable $y$ which occurs freely in $\phi[\frac{y}{x}]$ at those occurences where $x$ occurs freely in $\phi$.
    \item Recall that, since we are considering a language with only one relation symbol $\in$ and no other non-logical symbols, $t$ will just be a variable. Thus, an $r$-realizer for $\phi[\frac{t}{x}]\rightarrow\exists{x}\phi$ is the same as an $r$-realizer for $\forall{t}(\phi[\frac{t}{x}]\rightarrow\exists{x}\phi)$. Such an $r$-realizer works as follows: Given a code $c_{t}$ for a set, and then an $r$-realizer $r\Vdash_{r}\phi[\frac{t}{x}]$, it is easy to recognize an OTM-program $Q$ that recognizes the pair $(c_{t},r)$ (the input just needs to be compared to the oracle). Thus, $(Q,\rho)\Vdash_{r}\exists{x}\phi$. 
    
    \item As a true atomic formula, this, by definition of $\Vdash_{r}$, realized by any, say $\emptyset$. 
    \item Let codes $c_{x}$, $c_{y}$ for sets be given. Given an $r$-realizer $r\Vdash_{r}x=y$, an $r$-realizer for $\phi[\frac{s}{x}]\rightarrow\phi[\frac{t}{x}]$ just an OTM-program which, on input $s\Vdash_{r}\phi[\frac{s}{x}]$, recognizes $s$. Since $x=y$ is atomic, its $r$-realizability implies its truth, so that the sets coded by $c_{x}$ and $c_{y}$ will actually be the same. Hence, any $r$-realizer for $\phi[\frac{s}{x}]$ will also be one for $\phi[\frac{t}{x}]$.
\end{enumerate}
\end{proof}

\begin{lemma}{\label{deduction rules}}
Assume that $\Phi$ is set of $\in$-formulas such that all elements of $\Phi$ are $r$-realizable. 
Assume further that $\phi$ follows from $\Phi$ by one of the following deduction rules (cf. \cite{Aczel}, Def. 2.5.1). Then $\phi$ is $r$-realizable.
\begin{enumerate}
    \item $\{\phi,\phi\rightarrow\psi\}\vdash\psi$ 
    \item $\{\psi\rightarrow\phi[\frac{x}{y}]\}\vdash(\psi\rightarrow\forall{x}\phi)$, where $y$ is free for $x$ in $\phi$ and has no free occurences in $\phi$, $\psi$.
    \item $\{\phi[\frac{x}{y}]\rightarrow\psi\}\vdash(\exists{x}\phi\rightarrow\psi)$,  where $y$ is free for $x$ in $\phi$ and has no free occurences in $\phi$, $\psi$.
\end{enumerate}
\end{lemma}
\begin{proof}
For the sake of simplicity, we pick a realizer $r_{\phi}$ for each $\phi\in\Phi$. In fact, a realizer of $\phi$ can be uniformly OTM-computed from the realizers for the formulas involved in the deduction in each case. Since the proofs work the same whether parameters are allowed to be ordinals, sets of ordinals, or not at all, we suppress all mentioning of parameters.
    \begin{enumerate}
        \item $r_{\phi\rightarrow\psi}$ will be a program $P$ that recognizes a realizer for $\psi$ relative to any (code for a) realizer for $\phi$. Thus, $P^{r_{\phi}}$ will recognize an $r$-realizer for $\psi$, so $\psi$ is $r$-realizable.
        \item $r_{\psi\rightarrow\phi[\frac{x}{y}]}$ is a program $P$ that recognizes a $r$-realizer for $\phi[\frac{x}{y}]$ relative to any $r$-realizer of $\psi$. What we need is a program $Q$ that, given an $r$-realizer for $\psi$, recognizes a realizer for $\forall{x}\phi$. So let an $r$-realizer $r_{\psi}$ for $\psi$ be given. Then $P^{r_{\psi}}$ will recognize an $r$-realizer $r$ for $\phi[\frac{x}{y}]$. By the definition of $r$-realizers for formulas with free variables, $r$ is a realizer or $\forall{x}\phi$.
        \item Let $r$ be an $r$-realizer for $\phi[\frac{x}{y}]\rightarrow\psi$. Since $\phi[\frac{x}{y}]\rightarrow\psi$ contains a free variable, this means that $r\Vdash_{r}\forall{x}(\phi[\frac{x}{y}]\rightarrow\psi)$; that is, $r=(P,\rho)$, where $P$ is an OTM-program and $\rho$ is a parameter such that, for all codes $c$ for a set $b$, $P^{\rho,c}$ recognizes an $r$-realizer for $\phi(b)\rightarrow \psi$. What we need is a program $Q$, along with a parameter $\rho^{\prime}$ such that, for all $r$-realizers $u\Vdash_{r}\exists{x}\phi$, $Q^{u,\rho^{\prime}}$ recognizes an $r$-realizer for $\psi$.

        $Q$ will work as follows: Let an $r$-realizer $(c,t)$ for $\exists{x}\phi$ be given; that is, $c$ is an ordinal code for a set $b$ and $t\Vdash_{r}\phi(b)$. $Q$ will now recognize a pair $(s,s^{\prime})$ such that $s^{\prime}\Vdash_{r}\phi(b)\rightarrow\psi$ and $s\Vdash_{r}\psi$. 

        This is done as follows: By definition, $P^{\rho,c}$ will recognize a realizer $r^{\prime}$ for $\phi(b)\rightarrow\psi$. Let $s^{\prime}=r^{\prime}$. Given a pair $(p,q)$ as our input, we can thereby check the second component. If it turns out that $q\neq r^{\prime}$, we halt with output $0$. Otherwise, $q=r^{\prime}$ will consist of a program $S$ and a parameter $\rho^{\prime\prime}$ such that $S^{\rho^{\prime\prime},v}$ recognizes an $r$-realizer for $\psi$ whenever $v$ is an $r$-realizer for $\phi(b)$. In particular then, $S^{\rho^{\prime\prime},t}$ will recognize an $r$-realizer $v^{\prime}$ for $\psi$. Let $p=v^{\prime}$. Then the procedure just described recognizes $(p,q)=(v^{\prime},r^{\prime})$ in the input $(c,t)\Vdash_{r}\exists{x}\phi$, and we will have $p\Vdash\psi$. Hence the procedure is as desired.              
    \end{enumerate}
In each case, it is easy to see that the method for obtaining the desired $r$-realizer is effective in the given $r$-realizers for the assumptions.
\end{proof}




\begin{thebibliography}{}
\bibitem{Aczel} P. Aczel, M. Rathjen. CST Book Draft. Available online: \url{https://www1.maths.leeds.ac.uk/~rathjen/book.pdf}
\bibitem{Ca2018} M. Carl. Effectivity and reducibility with ordinal Turing machines. Computability, vol. 10(4) (2021)
\bibitem{Ca2024} M. Carl. Almost sure OTM-realizability. In: In: L. Patey et al. (eds). Twenty Years of Theoretical and Practical Synergies. CiE 2024. Lecture Notes in Computer Science, vol 14773. Springer, Cham.\url{ https://doi.org/10.1007/978-3-031-64309-5_13} (2024)
\bibitem{CarlBuch} M. Carl. Ordinal Computability. An Introduction to Infinitary Machines. De Gruyter Berlin/Boston (2019)
\bibitem{LoMe} M. Carl. The lost melody phenomenon. In: S. Geschke et al. (eds.). Infinity, Computability and Metamathematics. Festschrift celebrating the 60th birthdays of Peter Koepke and Philip Welch. College Publications (2014)
\bibitem{CarlNote} M. Carl. A Note on OTM-Realizability and Constructive Set Theories. Preprint.  arXiv:1903.08945v1 (2019)
\bibitem{CGP} M. Carl, L. Galeotti, R. Passmann. Realisability for Infinitary Intuitionistic Set Theory. Annals of Pure and Applied Logic, Vol. 174(6) (2023)
\bibitem{CGP2} M. Carl, L. Galeotti, R. Passmann. Randomising Realisability.  In: L. De Mol, et al. (eds.). Connecting with Computability. CiE 2021. Lecture Notes in Computer Science, vol 12813. Springer, Cham. \url{https://doi.org/10.1007/978-3-030-80049-9_8}
\bibitem{CSW} M Carl, P. Schlicht, P. Welch. Recognizabile sets and Woodin cardinals: computation beyond the constructible universe. Annals of Pure and Applied Logic, vol. 169(4) (2018)
\bibitem{Dawson} B. Dawson. Ordinal Time Turing Computation. PhD thesis, Bristol (2009) (unpublished)
\bibitem{GLL} L. Galeotti, E. Lewis, B. L\"owe. Symmetry for transfinite computability. In: D. Vedova et al. (eds.). Unity of Logic and Computation. CiE 2023. Lecture Notes in Computer Science, vol 13967. Springer, Cham. \url{https://doi.org/10.1007/978-3-031-36978-0_6}
\bibitem{Hodges} W. Hodges. On the effectivity of some field constructions. Proceedings of the London Mathematical Society, vol. 3(1) (1976)
\bibitem{Kanamori} A. Kanamori. The Higher Infinite. Springer Berlin Heidelberg New York (2003)
\bibitem{Kleene} S. Kleene. On the interpretation of intuitionistic number theory. Journal of Symbolic Logic, vol. 10(4) (1945)
\bibitem{Koepke} P. Koepke. Turing computations on ordinals. Bulletin of Symbolic Logic, vol. 11(3) (2005)
\bibitem{Krapf} R. Krapf. Class forcing and second-order arithmetic. Phd thesis, Bonn (2017)
\bibitem{Kunen} K. Kunen. Set Theory. An Introduction to Independence Proofs. Elsevier Amsterdam (1980)
\bibitem{Levy} A. Levy. Basic Set Theory. Springer Berlin Heidelberg (1979)
\bibitem{Rathjen} M. Rathjen. From the weak to the strong existence property. Annals of Pure and Applied Logic, vol. 163(10) (2012)
\bibitem{P} R. Passmann. The first-order logic of CZF is intuitionistic first-order logic. The Journal of Symbolic Logic, vol. 89(1) (2024)
\end{thebibliography}
\end{document}